\documentclass{amsart}
\usepackage{amsmath,amsthm}
\usepackage{amsfonts,amssymb}
\usepackage{mathrsfs}
\usepackage{enumerate}
\usepackage{verbatim}

\usepackage{multirow}
\usepackage{booktabs}

\usepackage[numbers,sort,compress]{natbib}

\newtheorem{thm}{Theorem}[section]

\newtheorem{lem}[thm]{Lemma}

\newtheorem{rem}{Remark}

\numberwithin{equation}{section}

\newcommand{\al}{\alpha}

\def\vz{\varepsilon}

\def\lz{\lambda}

\def\({\Bigl(}
\def \){ \Bigr)}

\newcommand{\NN}{\mathbb{N}}

\newcommand{\bsj}{\oldsymbol{j}}

 \def\NN{{\mathbb N}}

 \def\RR{{\mathbb R}}

\def\bsj{{\bf j}}

\makeatletter
\@namedef{subjclassname@2020}{\textup{2020} Mathematics Subject Classification}
\makeatother

\begin{document}

\def\RR{\mathbb{R}}
\def\Exp{\text{Exp}}
\def\FF{\mathcal{F}_\al}

\title[] {A note about algebraic $(s,t)$-weak tractability of linear
tensor product problems in the worst-case setting}

\author{Zirong Liu} \address{ School of Mathematical Sciences, Capital Normal
University, Beijing 100048, China.} \email{2240501011@cnu.edu.cn}

\author{Heping Wang} \address{ School of Mathematical Sciences, Capital Normal
University, Beijing 100048, China.}\email{wanghp@cnu.edu.cn}


\begin{abstract}

This paper is devoted to discussing  the  linear tensor product
problems in the worst  case setting. We consider algorithms that
use finitely many evaluations of arbitrary continuous linear
functionals. We investigate  algebraic $(s,t)$-weak tractability
(ALG-$(s,t)$-WT) under the absolute error criterion  in the case
$\lambda_1>1$, where $\lambda_1$ is the square of the univariate
maximal singular value. We solve the problem by
 giving the necessary and sufficient conditions
for ALG-$(s,t)$-WT on univariate  singular values   and fill the
gap left open.
\end{abstract}

\maketitle
\input amssym.def

\section{Introduction and main results}

In the field of information-based complexity (IBC), the
computational hardness of solving a multivariate  problem $S =
\{S_d\}_{d \in \mathbb{N}}$ is measured by its information
complexity $n(\varepsilon, S_d)$, which is defined as the minimal
number of information evaluations required to approximate the
solution to within an accuracy $\varepsilon > 0$. Tractability of
multivariate problems studies how the information complexity
$n(\varepsilon, S_d)$  depends on  the dimension  $d$ and the
reciprocal of the accuracy $\vz$. There are two kinds of
tractability based on polynomial convergence and exponential
convergence. The algebraic tractability (ALG tractability)
describes how  $n(\vz, d)$ behaves as a function of $d$ and
$\vz^{-1}$, while the exponential tractability (EXP tractability)
does as a function of $d$ and $1+\ln\vz^{-1}$. Recently the  study
of ALG and EXP tractabilities has attracted much interest, and a
great number of interesting results
 have been obtained (see
\cite{ CW1, DKPW14, DLPW11, GW11, IKPW16a, KW19, KPW, NW08, NW10, NW12,
PP14, PPW17, S13, SW15, W19, X15} and the references therein).

This note is devoted to discussing ALG tractability of linear
tensor product problems in the worst case setting. For such
problems with respect to the absolute and normalized error
criteria, necessary and sufficient conditions for various notions
of ALG tractability on univariate singular values were obtained in
\cite{NW08,  S13, GW11, SW15}, and  those for  EXP tractability
were obtained in \cite{HKW19, PP14, PPW17}. However, we notice
that there is only one case left open. More precisely, for the
absolute error criterion and in the case  $\lz_1>1$,  necessary
and sufficient condition for
 ALG-$(s,t)$-WT  were not obtained, where $\lz_1$ is  the square of the
 univariate maximal
singular value.

  In this note we investigate  ALG-$(s, t)$-WT   for the absolute error criterion in the case  $\lz_1>1$. We  solve the
problem by giving a necessary and sufficient condition for
 ALG-$(s,t)$-WT and fill  the  gap left open.

\subsection{Linear tensor product problems.}

\

Linear tensor product problems are
important special cases of linear multivariate problems over
Hilbert spaces. A linear tensor product problem is defined as the $d$-folded
tensor product of a single univariate linear problem, see
\cite[Chapter 5.2]{NW08}.

We consider a compact linear operator $S_1 : H_1\to G_1$, where
$H_1$ and $G_1$ are two Hilbert spaces, and without loss of
generality we assume $\dim H_1=\infty$. Then $W_1 := S_1^* S_1:
H_1 \to H_1$ is a compact self-adjoint non-negative operator,
where $S^*$ is the adjoint operator of the linear operator $S$.
The eigenvalues of $W_1$ are $\lz_j,\ j\in\NN$, which satisfy
$$\lambda_1=\lambda_2=\cdots=\lambda_m>\lambda_{m+1}\ge\lambda_{m+2}\ge\dots\ge0,\, m\in \NN .$$
Since $S_1$ is compact, it follows that
$\lim\limits_{j\to\infty}\lambda_j=0$. Without loss of generality,
we assume $\lz_2>0$.

 For $d\in \Bbb N$, the tensor product operator $S_{d}$ is defined
 to be the $d$-fold tensor product of the
compact linear operator $S_1$, i.e.,
$$S_{d}=\bigotimes_{i=1}^d S_{1}: H_{d}\to G_d,$$
where $H_{d}$ and $G_d$ are the Hilbert spaces given by
$$H_{d}=\bigotimes_{i=1}^dH_{1},\ \
G_d=\bigotimes_{i=1}^d G_1.$$
 We say that $S = \{S_d\}_{d\in\NN}$ is a linear tensor
product problem. Clearly, $W_{d}:=S_{d}^*S_{d}: H_{d}\to H_{d} $
is a compact linear operator and the eigenvalues of $W_{d}$ are
$\lz_{d,\bsj},\ {\bsj\in\NN^d}$, satisfying
$$\lz_{d,\bsj}=\prod_{k=1}^d \lz_{j_k}.$$
Assume that $\{\lambda_{d,j}\}_{j\in\Bbb N}$ is the non-increasing
rearrangement of $\{\lz_{d,\bsj}\}_{\bsj\in \Bbb N^d}$. Then
$$\lz_{d,1}\ge \lz_{d,2}\ge \dots \ge 0, \ \ \
\lz_{d,1}=\lz_1^d,$$and $\{\lz_{d,j}\}_{j\in\Bbb N}$ is the
squares of the ordered singular values of $S_d$.

Consider the problem of approximating $S_d$ by algorithms
$A_{n,d}$ using at most $n$ information evaluations from the class
$\Lambda^{\rm all}$ consisting of all continue functionals on
$H_d$.
 The form of $A_{d,n}$ is
$$
A_{n,d}(f)=\phi_{n,d}(L_1(f),L_2(f),\dots,L_d(f)),
$$
where $L_j\in \Lambda^{\rm all }=H_d^*$ and $\phi_{n,d}:\RR^n \to
G_d$ is any mapping. The $n$-th minimal worst case error of $S_d$
is
$$e(n,d)=\inf_{A_{n,d}}\sup_{\substack{f \in H_{d}\\ \|f\|_{H_{d} } \le
1}}\|S_{d}(f)-A_{n,d}(f)\|_{G_d}.$$ For $n=0$, the initial error
$e(0,d)$ is denoted by
$$ e(0,d)=\sup_{\substack{  f \in H_{d} \\ \|f\|_{H_{d}} \le 1}}\|S_{d}(f)\|_{G_d}.$$
It is well known that for the linear tensor product problem
$S=\{S_d\}$ in the worst case setting, the $n$-th minimal worst
case error is related to the $(n+1)$-th  singular value of the
operator $S_d$, i.e.,
$$e(n,d)=\sqrt{\lambda_{d,n+1}} \ \ {\rm and }\ \  e(0,d)=\sqrt{\lz_{d,1}}=\lz_1^{d/2}.$$

 The information complexity for the normalized error criterion  (NOR)
or the absolute error criterion (ABS) is defined by
\begin{equation*}
n^X(\vz,S_{d})=
\min\{\,n\in \NN : e(n,d) \le  \varepsilon {\rm CRI}_d\,\},
\end{equation*}
where $X\in\{{\rm ABS,\ NOR}\}$, and
$${\rm CRI}_d=\left\{
\begin{array}{ll}
1 & \mbox{ if } X={\rm ABS},\\
e(0,d) & \mbox{ if } X={\rm NOR}.
\end{array}\right.$$
 It follows
$$n^{\rm ABS}(\vz,S_d)=n^{\rm NOR}({\vz}/{\lz_1^{d/2}},S_d),\ \ \ n^{\rm ABS}(\lz_1^{1/2}\vz,S_1)=n^{\rm NOR}(\vz,S_1).$$
 The information complexity can be rewritten
as
\begin{align}\label{1.1}
n^X(\vz,S_{d})&=\min\{\,n\in \NN : \lambda_{d,n+1}\le  \varepsilon^2 {\rm CRI}^2_d\,\}\notag\\
&= \big|\big\{(j_1,\dots,j_d) \in \NN^d : \prod_{k=1}^d\lz_{j_k}>
\varepsilon^2 {\rm CRI}^2_d \big\}\big|,
\end{align}where $| A |$
represents the cardinality of the set $A$.

Various notions of ALG or EXP tractability  are defined in terms
of how the information complexity $n^X(\vz,S_{d})$ depends on $d$
and $\vz^{-1}$, or alternatively on $d$ and $1+\ln \vz^{-1}$.  In
this note, we consider ALG-$(s,t)$-weak tractability
(ALG-$(s,t)$-WT). For fixed $s, t> 0$, the problem $S$ is said to
be ALG-$(s,t)$-weakly tractable
 if
\begin{equation*}
\lim\limits_{\varepsilon^{-1} + d \to \infty} \frac{\ln
n^{X}(\varepsilon, S_d)}{\varepsilon^{-s}+d^t} = 0.
\end{equation*}
Roughly speaking,  ALG-$(s, t)$-WT means that  the information
complexity is neither exponential in $d^t$, nor in $\vz^{-s}$.

Similarly, for fixed  $s, t> 0$, the problem $S$ is said to be
EXP-$(s,t)$-weakly tractable (EXP-$(s,t)$-WT) if
\begin{equation*}
\lim\limits_{\varepsilon^{-1} + d \to \infty} \frac{\ln
n^{X}(\varepsilon, S_d)}{(1+\ln\varepsilon^{-1})^s +d^t} = 0.
\end{equation*}

\subsection{Main results}\label{Main Results}

\

For the linear tensor product problem $S$ in the worst case
setting, the information complexity $n^{X}(\vz,S_d )$ is fully
determined by the univariate singular values  of $S_1$. For NOR or
ABS, Siedlecki and Weimar in \cite[Theorems 3.9 and 3.10]{SW15}
gave the necessary and sufficient conditions for ALG-$(s,t)$-WT.
However, for ABS and in the case $\lambda_1>1$, they provided
separate necessary and sufficient conditions for ALG-$(s,t)$-WT.
See the following ALG-$(s,t)$-WT results of the linear tensor
product problem $S$ with $\lz_2>0$ for ABS and for the class
$\Lambda^{\rm all}$.

    \begin{itemize}
        \item Let $\lambda_1<1$. Then $S$ is ALG-$(s,t)$-WT if and only if
            \begin{equation}\label{1.2}
                \lim_{n\to\infty}\frac{\lambda_n}{\ln^{-2/s}n}=0.
            \end{equation}
        \item Let $\lambda_1=1$ and
            \begin{itemize}
                \item[(1)] assume that $m=1$. Then $S$ is ALG-$(s,t)$-WT if and only if
                    \begin{equation*}
                        \lim_{n\to\infty}\frac{\lambda_n}{\ln^{-2/s}n}=0.
                    \end{equation*}
                \item[(2)] assume that $m>1$. Then $S$ is ALG-$(s,t)$-WT if and only if
                    \begin{equation*}
                        t>1
                        \quad \text{and} \quad
                        \lim_{n\to\infty}\frac{\lambda_n}{\ln^{-2/s}n}=0.
                    \end{equation*}
            \end{itemize}
        \item Let $\lambda_1>1$ and define $S_1':=\frac{1}{\sqrt{\lambda_1}} S_1$. Then ALG-$(s,t)$-WT of $S$ implies
    \begin{equation*}
        t>1
        \quad \text{and} \quad
        \lim_{\vz^{-1}+d\to\infty} \frac{\max_{\ell=1,\ldots,d}\limits \left[ \ell \cdot \ln n^{\mathrm{abs}}\!\left( ( \vz / \lambda_1^{d/2} )^{1/\ell},\, S'_1 \right) \right]}{\vz^{-s}+d^{t}}
        = 0.
    \end{equation*}
        Moreover, the conditions
    \begin{equation*}
        \quad t>1
        \quad \text{and} \quad
        \lim_{\vz^{-1}+d\to\infty} \frac{\ln d \cdot \max_{\ell=1,\dots,d}\limits \left[ \ell \cdot \ln n^{\mathrm{abs}}\!\left( ( \vz /\lambda_1^{d/2} )^{1/\ell},\, S'_1 \right) \right]}{\vz^{-s}+d^{t}}
        = 0
    \end{equation*}
    are sufficient for $S$ to be ALG-$(s,t)$-WT.
    \end{itemize}

Although Siedlecki and
Weimar did not establish necessary and sufficient conditions on
 $\{\lz_j\}$ for ALG-$(s,t)$-WT under the ABS setting with $\lz_1>1$, they
 gave in \cite[Lemma 3.12]{SW15} that  a necessary condition for ALG-$(s,t)$-WT is
\begin{equation*}
\lim_{j\rightarrow\infty}\frac{(\ln\frac{1}{\lambda_j})^t}{\ln
j}=\infty.
\end{equation*}
Clearly, this necessary condition is much stronger than the decay
condition \eqref{1.2} which characterizes ALG-$(s,t)$-WT in the
case $\lz_1\le1$ or for NOR.

In  this note, we investigate ALG-$(s,t)$-WT for ABS in the case
$\lz_1>1$ and obtain a necessary and sufficient condition for
 ALG-$(s,t)$-WT. We show that  the above necessary condition
 together with
 the condition $t>1$ are also sufficient for ALG-$(s,t)$-WT. Our
 main result can be formulated as follows.

\begin{thm}\label{thm1} Consider a linear tensor product problem $S=(S_d)_{d\in\NN}$ with $\lambda_2>0$ in the
worst case setting for the absolute error criterion and for the
class $\Lambda^{\rm all}$.
 Assume that $\lambda_1>1$. Then $S$ is ALG-$(s,t)$-WT if and only if
 $t>1$ and
\begin{equation}\label{1.3}
\lim_{j\rightarrow\infty}\frac{(\ln\frac{1}{\lambda_j})^t}{\ln
j}=\infty.
\end{equation}
\end{thm}

\begin{rem}It is interesting
to see that ALG-$(s,t)$-WT only involves the parameter $t$ in the
case $\lz_1>1$ and for ABS. The authors in \cite{PPW17}
investigated EXP-tractability of the linear tensor product problem
 $S=\{S_d\}_{d\in\NN}$ with $\lambda_2>0$ in the
worst case setting  for the class $\Lambda^{\rm all}$. They
obtained that in the case $\lz_1> 1$ and for ABS, $S$ is
EXP-$(s,t)$-WT if and only if $t>1$ and
\begin{equation*}
\lim_{j\rightarrow\infty}\frac{(\ln\frac{1}{\lambda_j})^{\min(s,t)}}{\ln
j}=\infty.
\end{equation*}By Theorem \ref{thm1}, we obtain that  for ABS and in the case $\lz_1> 1$, $S$ is
ALG-$(s,t)$-WT if and only if  $S$ is ALG-$(t,t)$-WT, and  if and
only if  $S$ is EXP-$(s,t)$-WT with $s>t$.
\end{rem}

The paper is organized as follows. In Section 2 we give some
preliminaries in the worst case setting. In Section 3, we give the
proof of Theorem \ref{thm1}.

\section{Premilinaries}\

For $\varepsilon \in (0,1)$ and
$\lim\limits_{j\to\infty}\lambda_j=0$, define
\begin{equation}
 j(\varepsilon) :=n^{\rm NOR}(\vz,S_1)=  \max\{j \in \NN : \lambda_j > \varepsilon^2\lambda_1\}.
\end{equation}
We put $j(\vz)=0$
for $\vz\ge1$. Then $j(\vz)$ is well defined and always finite.   Furthermore, $j(\vz)$ goes
to infinity if and only if all $\lambda_j$'s are positive.

Consider the normalized problem $S'=\{S'_d\}$. Let
$S'_1=\frac{1}{\sqrt{\lambda_1}}S_1$. The ordered eigenvalues
$\lz'_j,\ {j\in \NN}$ of $W'_1:= (S'_1)^ * \cdot S'_1$ satisfy
$\lz'_j=\frac{\lz_j}{\lz_1}$ and
$$1=\lz'_1=\cdots=\lz'_m>\lz'_{m+1}\ge\lz'_{m+2}\ge\dots\ge0.$$
Let $S'_d=\bigotimes\limits_{i=1}^{d}S'_1$ be the $d$-fold tensor
product operator of $S_1'$. Then the eigenvalues of
$W'_{d}:=(S'_{d})^*\cdot S'_{d}$ are
$$\lz'_{d,\bsj}=\prod_{k=1}^d \lz'_{j_k},\ \bsj\in \Bbb N^d.$$
Assume that $\{\lambda'_{d,j}\}_{j\in\Bbb N}$ is the
non-increasing rearrangement of $\{\lz'_{d,\bsj}\}_{\bsj\in\Bbb
N^d}$. That is
$$1=\lz'_{d,1}\ge \lz'_{d,2}\ge \dots \ge 0.$$
The information complexity of $S'$ is given by
 $$n^{X}(\vz,S'_d):= \Big|\big\{\bsj \in \NN^d : \lambda'_{j_1} \lambda'_{j_2}\cdots
  \lambda'_{j_d}> \varepsilon^2 {\rm CRI}^2_d  \big\}\Big|.$$
It is obvious that
$$n^{\rm ABS}(\vz,S'_d)=n^{\rm NOR}(\vz,S'_d)=n^{\rm NOR}(\vz,S_d), \ n^{\rm ABS}(\vz,S'_1)=n^{\rm NOR}(\vz,S_1)=j(\vz).$$

\begin{lem}\label{lem1}
 Let $\{\lambda_j\}_{j\in\NN}$ be a non-increasing sequence and $\lim\limits_{j\rightarrow\infty}\lambda_j=0$. Then \eqref{1.3}   holds if and only if
 \begin{equation}\label{2.2}
 \lim_{\vz\to 0}\frac{\ln j(\vz)}{(\ln\vz^{-1})^t}=0.
\end{equation}
\end{lem}

\begin{proof}If $\lz_{j}=0$ for
some $j_0\in\NN$, then $\lambda_j=0$ for $j\ge j_0$ and hence
$j(\vz)\le j_0$. It follows that \eqref{1.3} and \eqref{2.2} hold in this case. Without loss of generality, we
assume that $\{\lz_j\}$ is a positive sequence.

We first show that \eqref{2.2} implies \eqref{1.3}. Assume that
\eqref{2.2} holds. For every $\eta\in(0,1)$, there exists an
$\vz_0\in(0,1)$ such that
$$\frac{\ln j(\vz)}{(\ln \vz^{-1})^t}<\eta  \ \ \ {\rm for}\  \vz\in(0,\vz_0).$$
This implies
$$j(\vz)<e^{\eta (\ln\vz^{-1})^t} \le \lceil e^{\eta (\ln\vz^{-1})^t}\rceil :=v(\vz),$$
where $\lceil x\rceil$  is the ceiling of a real number $x$.  From the definition of $j(\vz)$, we
have $j(\vz)+1\le v(\vz)$ and
$$\lambda_{v(\vz)}\le \lambda_{j(\vz)+1}\le \vz^2\lz_1.$$
Let $v(\vz)=k$. If we vary $\vz\in(0,\vz_0)$, then $k$ can take the
values
$$k= \lceil e^{\eta (\ln\vz_0^{-1})^t} \rceil,\lceil e^{\eta (\ln\vz_0^{-1})^t}\rceil+1,\cdots.$$
Since $k\le e^{\eta (\ln\vz^{-1})^t}+1$, we get
$$\vz^{2}\le e^{-2\big(\frac{\ln (k-1)}{\eta}\big)^{1/t}}.$$
Hence, $$\lambda_{k}\le \vz^2\lz_1\le  e^{-2\big(\frac{\ln
(k-1)}{\eta}\big)^{1/t}}\cdot\lz_1\ \  \ {\rm for}\  k\ge \lceil
e^{\eta (\ln\vz_0^{-1})^t}\rceil. $$  This  yields
  \begin{align*}\frac{(\ln\frac{1}{\lambda_{k}})^t}{\ln k}\
  &\ge\frac{\big((\frac{ 2\ln(k-1)}{\eta})^{1/t}-\ln\lz_1\big)^t}{\ln k}\\&\ge
  \frac{\frac{2\ln(k-1)}{\eta}\big(1-\frac{\ln \lz_1\cdot t}{(2 \ln(k-1)/\eta)^{1/t}}\big)}{\ln k}\to \frac{2}{\eta}
  \ \ \ {\rm as \ }k\to\infty,\end{align*}where in
  the last inequality we used the fact that
$(a-b)^t\ge a^t(1-\frac{b}{a}t)$ for $a\ge b> 0$ and $t>1$. Since $\eta$ can be arbitrary small, we obtain \eqref{1.3}.

Next, if $\eqref{1.3}$ holds, then for every $M>0$ there exists a
$J>0$ such that
$$\big(\ln\frac{1}{\lambda_j}\big)^t\geq M\ln j \ \ \ {\rm for}\ j\geq J.$$
This implies $\lambda_j\le e^{-(M\ln j)^{\frac{1}{t}}}$. We solve
the inequality $\vz^2\lz_1\ge  e^{-(M\ln j)^{\frac{1}{t}}},$ and obtain the solution
$$j\ge e^{\frac{(2\ln\vz^{-1}-\ln\lz_1)^t}{M}}.$$
This concludes that if $$j\ge \max\{J,\
e^{\frac{(2\ln\vz^{-1}-\ln\lz_1)^t}{M}}\},$$ then $$\lz_j\le
e^{-(M\ln j)^{\frac{1}{t}}}\le \vz^2\lz_1.$$
 From the definition of $j(\vz)$, we obtain
$$j(\vz)\le \max\{J,\  e^{\frac{(2\ln\vz^{-1}-\ln\lz_1)^t}{M}}\}.$$
Taking the logarithm yields that for every $\delta\in(0,s)$,
$$\frac{\ln j(\vz)}{(\ln \vz^{-1})^t}\le \frac{\max\{\ln J,\ \frac{(2\ln\vz^{-1}-\ln\lz_1)^t}{M}\}}{(\ln \vz^{-1})^t}\to
\frac{2^t}M
  \ \ \ {\rm as \ }\vz\to 0.$$
Since $M$ can be arbitrary large, we obtain \eqref{2.2}.

Lemma \ref{lem1} is proved.
\end{proof}

The number of $k$-element subsets of an $n$-element set is called the number of combinations of $n$ objects taken
$k$ at a time, and is denoted by
$$\binom{n}{k}=\frac{n!}{k!(n-k)!},$$
and the number of permutations of $n$ objects taken $k$ at a time, denoted $P(n,k)$, is given by
 $$P(n,k)=\binom{n}{k}\cdot k!=\frac{n!}{(n-k)!}.$$

\begin{lem}\label{lem2}
We have
\begin{equation}\label{2.3}
n^{\rm NOR}(\vz,S_d')\le
P\big(d,{a_{\vz}(d)}\big)m^{d-a_{\vz}(d)}\prod_{k=1}^{a_{\vz}(d)}j(\vz^{\frac{1}{k}}),
\end{equation}
where  $a_{\vz}(d):=\min\big\{d,
\big\lceil\frac{\ln\vz^{-2}}{\ln(\lambda'_{m+1})^{-1}}\big\rceil-1\big\}$.
\end{lem}
\begin{proof}
For $\vz\in(0,1)$, set $$A(\vz)=\big\{\bsj \in \NN^d :
\lambda'_{j_1} \lambda'_{j_2}\cdots  \lambda'_{j_d}> \varepsilon^2
\big\}=\big\{\bsj \in \NN^d : \lambda_{j_1} \lambda_{j_2}\cdots
\lambda_{j_d}> \varepsilon^2 \lz_1^{d}\big\}.$$ By \eqref{1.1}
we get
$$n^{\rm NOR}(\vz,S'_d)=n^{\rm NOR}(\vz,S_d)= \big|A(\vz)\big|.$$
 We have $\lz'_j=\lz_j/\lz_1$ and
$$1=\lz'_1=\cdots=\lz'_m>\lz'_{m+1}\ge\lz'_{m+2}\ge\cdots\ge0.$$

If $\lz'_{m+1}=0$, then $$a_{\vz}(d)=0, \ \ n^{\rm
NOR}(\vz,S'_d)=m^d,$$ and \eqref{2.3} holds. Otherwise, suppose
that $\lz'_{m+1}>0$. Set $$u=|\{k\in [d] : j_k>m\}|$$ for $\bsj\in
A(\vz)$. Then
$$(\lz'_{m+1})^{u} \ge\prod_{k=1}^d\lambda'_{j_k}>\vz^2.$$
It implies that
$$u\le \big\lceil\frac{\ln\vz^{-2}}{\ln(\lambda'_{m+1})^{-1}}\big\rceil-1.$$
Then we have
\begin{equation*}
u\le \min\big\{d,
\big\lceil\frac{\ln\vz^{-2}}{\ln(\lambda'_{m+1})^{-1}}\big\rceil-1\big\}:=a_{\vz}(d),
\end{equation*}i.e., for $\bsj\in
A(\vz)$ there are at most $a_{\vz}(d)$ indices $j_k$ which satisfy
$j_k>m$. Hence there are at least $d-a_{\vz}(d)$ indices $j_k$
which satisfy $j_k\le m$ and hence $\lz_{j_k}'=1$.  There are
$\binom{d}{a_{\vz}(d)}m^{d-a_{\vz}(d)}$ ways to select $d-a_{\vz}(d)$
indices $j_k$ that satisfy $j_k\le m$.

For the remaining $a_{\vz}(d)$ indices, let $j_{1,\max}$ be the
largest index of the eigenvalues in the product
$\prod\limits_{k=1}^d\lambda'_{j_k}$. Since
$$\lambda'_{j_{1,\max}}\ge   \prod\limits_{k=1}^d\lambda'_{j_k}>\vz^2,$$
we obtain that
$$j_{1,\max}\le \max\big\{j:\lambda'_j>{\vz^2}\big\}
=j(\vz).$$ Let $j_{k,\max}$ be the $k$-th largest index of the
eigenvalues in the product $\prod\limits_{k=1}^d\lambda'_{j_k},$
$k=2,\cdots,a_\vz(d)$. Since
$$ (\lambda'_{j_{k,\max}})^k\ge \prod_{i=1}^k \lambda'_{j_{i,\max}}\ge  \prod_{k=1}^d\lambda'_{j_k}
>\vz^2,$$
we obtain
\begin{equation*}
j_{k,\max}\le \max\big\{j:\lambda'_j>\vz^{\frac{2}{k}}\big\} =j(\vz^{\frac{1}{k}})
.
\end{equation*}
Then for the remaining $a_{\vz}(d)$ indices, if $k=1,\dots,
a_\vz(d)$, only the $k$-th largest indices
$$j_{k,\max}\in\big\{1, 2, \dots, j(\vz^{\frac{1}{k}})\big\},\
$$ may belong to $A(\vz)$ and there are at most $j(\vz^{\frac{1}{k}})$ ways to select $j_{k,\max}$.
Also there are  $a_{\vz}(d)!$ arrangements for these $a_{\vz}(d)$
indices. Therefore, we obtain
$$n^{\rm NOR}(\vz,S_d')\le
\binom{d}{a_{\vz}(d)}m^{d-a_{\vz}(d)}
\Big(\prod_{k=1}^{a_{\vz}(d)}j(\vz^{\frac{1}{k}})\Big)
a_{\vz}(d)!.
$$
Hence \eqref{2.3} holds. Lemma \ref{lem2} is proved.
\end{proof}

\section{Proofs of Theorems \ref{thm1}}

\

\noindent{\it Proof of Theorem \ref{thm1}.}

 {\bf Necessity.} The necessity for ALG-$(s,t)$-WT was obtained in \cite{SW15}. We give the proof for
 the sake of readers' convenience.  Assume that ALG-$(s,t)$-WT holds for ABS. Since $\lz_1>1$,
 from \cite[Theorem 5.5]{NW08}, we know that $S$ suffers from the curse of
dimensionality, and   there exist  constants $c>0$ and $a>0$, such
that for all $\vz\in(0,1)$,
$$n^{\rm ABS}(\vz, S_d)\ge c(1+a)^d.$$
Hence we have $t>1$. Otherwise
$$\frac{\ln n^{\rm ABS}(1/2,S_d)}{2^s+d^t}\ge \frac{\ln c+d\ln(1+a)}{2^s+d^t}$$
does not tend to zero as $d\to\infty$ which leads to a
contradiction. Consider the normalized problem $S'=\{S'_d\}$, then
we have $S'_d=\frac{1}{\sqrt{\lambda^d_1}}S_1$. The information
complexity of the problem $S$ can be rewritten as
 $$n^{\rm ABS}(\vz,S_d):= \Big|\big\{\bsj \in \NN^d : \lambda'_{j_1} \lambda'_{j_2}\cdots
 \lambda'_{j_d}> \frac{\varepsilon^2}{\lz_1^d}  \big\}\Big|=n^{\rm ABS }(\vz/\lambda_1^\frac{d}{2},S'_d).$$
Then we have
 \begin{align*}
 n^{\rm ABS}(\vz,S_d)&\ge \big|\big\{j_1 \in \NN : \lambda'_{j_1}> \frac{\varepsilon^2}{\lz_1^d} \big\}\big|= n^{\rm ABS}\Big(\vz/\lambda_1^{\frac{d}{2}},S'_1\Big)=j(\vz/\lambda_1^{\frac{d}{2}}).
 \end{align*}
 Taking $\vz=\frac{1}{2}$, we have
 $$0=\lim_{\vz^{-1}+d\to \infty}\frac{\ln n^{\rm ABS}(\vz,d)}{\vz^{-s}+d^t}=\lim_{d\to \infty}\frac{\ln n^{\rm ABS}(\frac{1}{2},d)}{2^s+d^t}\ge \lim_{d\to \infty}\frac{\ln j(1/(2\lambda_1^{\frac{d}{2}}))}{d^t}.$$
 Thus
 $$\lim_{d\to \infty}\frac{\ln j(1/(2\lambda_1^{\frac{d}{2}}))}{d^t}=0.$$
  Let $\vz=1/(2\lambda_1^{\frac{d}{2}})$, we have $d=\frac{\ln1/2+\ln\vz^{-1}}{\ln \lz_1}$.
This
 concludes that
$$\lim_{\vz\to 0}\frac{\ln j(\vz)}{(\ln \vz^{-1})^t}=0,$$and by Lemma
\ref{lem2} we get \eqref{1.3}.

{\bf Sufficiency.} Assume that \eqref{1.3} holds and $t>1$, we
show that $S$ is ALG-$(s,t)$-WT for ABS. Consider the normalized
problem $S'=(S'_d)$,
$$n^{\rm ABS}(\vz,S_d)=n^{\rm NOR}(\vz/\lambda_1^\frac{d}{2},S'_d).$$
 By Lemma \ref{lem2}, we have
\begin{align}
n^{\rm NOR}(\vz,S'_d)\le
P\big(d,{a_{\vz}(d)}\big)m^{d-a_{\vz}(d)}\prod_{k=1}^{a_{\vz}(d)}j(\vz^{\frac{1}{k}})\le
d!\, m^d \prod_{k=1}^{d}j(\vz^{\frac{1}{k}}),\label{3.1}
\end{align}
where
\begin{equation*}\
a_{\vz}(d):=\min\big\{d,\big\lceil\frac{\ln\vz^{-2}}{\ln(\lambda'_{m+1})^{-1}}\big\rceil-1\big\}.
\end{equation*}
For $x\in(0,1)$, let $\ln j(x)=b(x) (\ln x^{-1})^t$. Since
\eqref{1.3} holds,  by Lemma \ref{lem1} we have \eqref{2.2}. Then
there exists a constant number $M$ such that
$$\lim\limits_{x\to 0^+}b(x)= 0, \ \ {\rm and}\ \  0\le b(x)\le M \ {\rm for\ all }\ x\in(0, 1).$$
Taking the logarithm in \eqref{3.1} we get
\begin{align}\label{3.2}
{\ln n^{\rm ABS}(\vz,S_d)}&\le {\ln d!+d\ln m+\sum_{k=1}^{d}\ln
j\big((\vz/\lz_1^{\frac{d}{2}})^{\frac{1}{k}}\big)}\notag \\ &\le
d\ln d +d\ln m+
\sum_{k=1}^{d}b\big((\vz/\lz_1^{\frac{d}{2}})^{\frac{1}{k}}\big)\Big(\frac{1}{k}(\ln\vz^{-1}+\frac{d}{2}\ln
\lz_1)\Big)^t
\notag \\
&\le d\ln (md)
+\sum_{k=1}^{d}b\big((\vz/\lz_1^{\frac{d}{2}})^{\frac{1}{k}}\big)\frac{1}{k^t}\Big((2\ln
\vz^{-1})^t+d^t(\ln \lz_1)^t\Big)\notag\\ &\le d\ln (md)
+2^tM\sum_{k=1}^{d}\frac{(\ln \vz^{-1})^t}{k^t} +(\ln
\lz_1)^td^t\sum_{k=1}^{d}b\big((\vz/\lz_1^{\frac{d}{2}})^{\frac{1}{k}}\big)\frac{1}{k^t}\notag\\
&=: I_1+2^tM I_2+(\ln \lz_1)^t I_3,
\end{align}
where in the third inequality we used the inequality $(a+b)^t\le
2^t(a^t+b^t)$ for $a,b\ge 0,\ t>1$. Since $t>1$, we have
\begin{align*}
\lim_{\varepsilon^{-1}+d \rightarrow \infty} \frac{I_1}{
\varepsilon^{-s}+d^t }=\lim_{\varepsilon^{-1}+d \rightarrow \infty}
\frac{d\ln (md)}{\varepsilon^{-s}+d^t } = 0,
\end{align*}
and
\begin{align*}
\lim_{\varepsilon^{-1}+d \rightarrow \infty} \frac{I_2}{
\varepsilon^{-s}+d^t }=\lim_{\varepsilon^{-1}+d \rightarrow \infty}
\frac{\sum\limits_{k=1}^{d}\frac{1}{k^t}\cdot (\ln \vz^{-1})^t}{
\varepsilon^{-s}+d^t }\le \zeta(t) \lim_{\varepsilon^{-1}+d \rightarrow
\infty} \frac{(\ln \vz^{-1})^t}{ \varepsilon^{-s}+d^t } = 0,
\end{align*}
where $\zeta(t)=\sum_{k=1}^\infty\frac1{k^t}$ is  the Riemann-Zeta
function and is finite for $t>1$.

Since the series $\sum_{k=1}^{\infty}\frac{1}{k^t}$ is convergent
for $t>1$, we get for any $\delta>0$, there exists a constant
$N_\delta\in\NN$ which depends on $\delta$, such that
$$\sum_{k=N_\delta+1}^{\infty}\frac{1}{k^t}<\delta.$$
Then
\begin{align*}
I_3&=
\sum\limits_{k=1}^{N_\delta}b\big((\vz/\lz_1^{\frac{d}{2}})^{\frac{1}{k}}\big)
\frac{1}{k^t}\cdot d^t +\sum\limits_{k=N_\delta+1}^{\infty}
b\big((\vz/\lz_1^{\frac{d}{2}})^{\frac{1}{k}}\big)\frac{1}{k^t}\cdot d^t\\
     &\le \sum\limits_{k=1}^{N_\delta}
     b\big((\vz/\lz_1^{\frac{d}{2}})^{\frac{1}{k}}\big)\frac{1}{k^t}\cdot d^t+M\delta\cdot d^t.
\end{align*}
For the above $\delta>0$, since $\lim\limits_{x\to0}b(x)=0$, there
exists a $\delta_1\in(0,1)$ such that  for  $x\in (0,\delta_1)$,
it holds
$$b(x)<\delta.$$
Since $\lz_1>1$, we get  that
$$\lim_{d+\vz^{-1}\to\infty}\vz/\lz_1^{\frac{d}{2}}=0.$$
Hence  there exists a number $X>0$ such that
$\vz/\lz_1^{\frac{d}{2}}<\delta_1^{N_\delta}$ whenever
$d+\vz^{-1}>X$. It follows that
$$b\big((\vz/\lz_1^{\frac{d}{2}})^{\frac{1}{k}}\big)\le \delta, \ \ {\rm for }\ \ k=1, 2, \dots , N_\delta.$$
Hence$$\sum\limits_{k=1}^{N_\delta}
b\big((\vz/\lz_1^{\frac{d}{2}})^{\frac{1}{k}}\big)\frac{1}{k^t}
\cdot d^t \le \delta\zeta(t)\, d^t.$$ We obtain that  for
$d+\vz^{-1}>X$,
\begin{align*}
\frac{I_3}{\vz^{-s}+d^t }&=\frac{d^t\sum\limits_{k=1}^{d}b\big((\vz/\lz_1^{\frac{d}{2}})^{\frac{1}{k}}\big)\frac{1}{k^t}}
{\vz^{-s}+d^t }
\\ & \le \delta\zeta(t)
\cdot \frac{d^t}{\vz^{-s}+d^t }+M\delta\cdot\frac {d^t}{\vz^{-s}+d^t }\\
&\le \delta(\zeta(t)+M).
\end{align*}
Since $\delta$ can be arbitrary small, we conclude
$$\lim_{\varepsilon^{-1}+d \rightarrow \infty} \frac{I_3}{d^t +
\varepsilon^{-s}}=0.$$ Going back to \eqref{3.2} we have
\begin{align*}
\lim_{\varepsilon^{-1}+d \rightarrow \infty} \frac{\ln n^{\rm
ABS}(\varepsilon, S_{d})}{\varepsilon^{-s}+d^t } \le
\lim_{\varepsilon^{-1}+d \rightarrow \infty}\frac{I_1+2^tMI_2+
(\ln \lz_1)^tI_3}{\vz^{-s}+d^t }=0.
\end{align*}

 Theorem \ref{thm1} is proved. $\hfill\Box$


\begin{thebibliography}{99}

\bibitem{CW1} J. Chen, H. Wang. Average case tractability of multivariate approximation with Gaussian kernels.
Journal of Approximation Theory, 2019, 239: 51-71.

\bibitem{DKPW14} J. Dick, P. Kritzer, F. Pillichshammer, H. Wo\'zniakowski. Approximation of analytic functions in Korobov spaces.
Journal of Complexity, 2014, 30: 2-28.

\bibitem{DLPW11} J. Dick, G. Larcher, F. Pillichshammer, H. Wo\'zniakowski. Exponential convergence and tractability
of multivariate integration for Korobov spaces. Mathematics of Computation, 2011, 80: 905-930.

\bibitem{GW11} M. Gnewuch, H. Wo\'zniakowski. Quasi-polynomial tractability. Journal of Complexity, 2011, 27: 312-330.

\bibitem{HKW19} F. J. Hickernell, P. Kritzer, H. Wo\'zniakowski. Exponential tractability of linear
tensor product problems. In: D. R. Wood, J. DeGier, C. Praeger, T. Tao (eds.), MATRIX Annals. Cham: Springer, 2020.

\bibitem{IKPW16a} C. Irrgeher, P. Kritzer, F. Pillichshammer, H. Wo\'zniakowski.
Tractability of multivariate approximation defined over Hilbert spaces with exponential weights.
Journal of Approximation Theory, 2016, 207: 301-338.

\bibitem{KW19} P. Kritzer, H. Wo\'zniakowski. Simple characterizations of exponential tractability for linear multivariate
problems. Journal of Complexity, 2019, 51: 110-128.

\bibitem{KPW} P. Kritzer, F. Pillichshammer, H. Wo\'zniakowski. Exponential tractability of linear
 weighted tensor product problems in the worst-case setting for arbitrary linear functionals. {Journal of Complexity}, 2020, 61: 101501.

\bibitem{NW08} E. Novak, H. Wo\'zniakowski. \textit{Tractability of Multivariate Problems, Volume I: Linear Information}.
 Zurich: European Mathematical Society, 2008.

\bibitem{NW10} E. Novak, H. Wo\'zniakowski. \textit{Tractability of Multivariate Problems, Volume II: Standard Information for Functionals}.
Zurich: European Mathematical Society, 2010.

\bibitem{NW12} E. Novak, H. Wo\'zniakowski. \textit{Tractability of Multivariate Problems, Volume III: Standard Information for Operators}.
Zurich: European Mathematical Society, 2012.

\bibitem{PP14} A. Papageorgiou, I. Petras. A new criterion for tractability of multivariate problems. Journal of Complexity, 2014, 30: 604-619.

\bibitem{PPW17} A. Papageorgiou, I. Petras, H. Wo\'zniakowski. $(s,\ln^{\kappa})$-weak tractability of linear problems.
Journal of Complexity, 2017, 40: 1-16.

\bibitem{S13} P. Siedlecki. Uniform weak tractability. Journal of Complexity, 2013, 29: 438-453.

\bibitem{SW15} P. Siedlecki, M. Weimar. Notes on $(s,t)$ weak tractability: a refined classification of problems with
(sub)exponential information complexity. Journal of Approximation Theory, 2015, 200: 227-258.

\bibitem{W19} H. Wang. A note about EC-$(s,t)$-weak tractability of multivariate approximation with analytic Korobov kernels.
 Journal of Complexity, 2019, 55: 101412.

\bibitem{X15} G. Xu. Exponential convergence-tractability of general linear problems in the average case setting.
Journal of Complexity, 2015, 31: 617-636.

\end{thebibliography}
\end{document}